\pdfoutput=1
\documentclass[a4paper,12pt]{article}

\usepackage{cmap}
\usepackage[T2A]{fontenc}
\usepackage[utf8]{inputenc} 
\usepackage[english]{babel}
\usepackage{amssymb,amsmath,amsthm}
\usepackage{hyperref}
\usepackage{ifthen}

\usepackage{navigator}
\embeddedfile{sourcefile}{mod3stsR2.tex}
\embeddedfile{pythonformula}{etc/sts3.py}
\hypersetup{pdfsubject={The number of the non-full-rank Steiner triple systems}, pdfauthor={Minjia Shi, Li Xu, Denis Krotov}}

\makeatletter
\def\@seccntformat#1{\csname the#1\endcsname.\ } 
\def\@biblabel#1{#1.} 
\makeatother

\date{}

\newif\ifNoRemark
\def\addtheorem#1#2#3#4{
\ifthenelse{\equal{#2}{}}{}%
{\ifthenelse{\expandafter\isundefined\csname the#2\endcsname}{\newcounter{#2}}{}}
\newenvironment{#1}[1][\global\NoRemarktrue]
{\par\addvspace{2mm plus 0.5mm minus 0.2mm}\noindent 
{\bf #3}\ifthenelse{\equal{#2}{}}{}%
{\refstepcounter{#2}{\bf ~\csname the#2\endcsname}}%
{\bf \vphantom{##1}\ifNoRemark.\ \else\ (##1).\fi}\begingroup #4}%
{\endgroup\par\addvspace{1mm plus 0.5mm minus 0.2mm}\global\NoRemarkfalse}
\expandafter\newcommand\csname b#1\endcsname{\begin{#1}}
\expandafter\newcommand\csname e#1\endcsname{\end{#1}}
}

\addtheorem{theorem}{thrm}{Theorem}{\sl}
\addtheorem{proposition}{thrm}{Proposition}{\sl}
\addtheorem{lemma}{thrm}{Lemma}{\sl}
\addtheorem{corollary}{thrm}{Corollary}{\sl}
\addtheorem{pro}{thrm}{Proposition}{\sl}
\addtheorem{example}{exmpl}{Example}{}
\addtheorem{remark}{rmrk}{Remark}{}

\makeatletter
\@addtoreset{thrm}{section}
\makeatother

\title{The number of the non-full-rank Steiner triple systems}
\author{Minjia Shi%
\thanks{Key Laboratory of Intelligent Computing and Signal Processing of Ministry of Education, School of Mathematics, Anhui Universit, Hefei, 230601, Anhui, China}\label{thx:2}
\and
Li Xu%
\footnotemark[2]
\and
Denis S. Krotov%
\thanks{Sobolev Institute of Mathematics, pr. Akademika Koptyuga 4, Novosibirsk 630090, Russia; e-mail: krotov@math.nsc.ru}%
}

\begin{document}
\maketitle
\begin{abstract}
The $p$-rank of a Steiner triple system $B$ is the dimension of the linear span of the set of characteristic vectors of blocks of $B$, over GF$(p)$. We derive a formula for the number of different Steiner triple systems of order $v$ and given $2$-rank $r_2$, $r_2<v$, and a formula for the number of Steiner triple systems of order $v$ and given $3$-rank $r_3$, $r_3<v-1$. 
Also, we prove that there are no Steiner triple systems of $2$-rank smaller than $v$ and, at the same time, $3$-rank smaller than $v-1$.
Our results extend previous work on enumerating Steiner triple systems
according to the rank of their codes, mainly by Tonchev, V.A.Zinoviev and
D.V.Zinoviev for the binary case and by Jungnickel and Tonchev for the
ternary case.
\end{abstract}

\section{Introduction}

A Steiner triple system STS$(v)$
is a finite set $S$ of cardinality $v$, whose elements are called points,
provided with a collection of $3$-subsets, called blocks, such that every
$2$-subset of $S$ is contained in one and only one block.
We assume that $S=\{1,\ldots,v\}$
and identify a block $b$ with its characteristic vector,
that is, the $v$-tuple of $0$s and $1$s having $1$s
in the coordinates numbered by the elements of $b$.
E.g., $(0,1,0,1,1,0,0) = \{2,4,5\}$ ($v=7$).
The dimension of the space over $\mathrm{GF}(p)$ spanned by the blocks
(to be exact, by their characteristic vectors)
of a Steiner triple system $T$
is called the $p$-rank of $T$.

A series of papers was devoted to the study of STS$(v)$ of deficient rank.
In 1978, Doyen, Hubaut, and Vandensavel \cite{DHV:1978} characterized the
classical examples --- the point-line designs in binary projective and ternary
affine spaces --- as the unique STS with these parameters which have minimal
$2$- and $3$-rank, respectively, and also proved that the $p$-rank of STS$(v)$, $v>3$, equals
 $v$ for every prime $p$ larger than $3$.
Later, Assmus in his fundamental paper \cite{Assmus:95}
gave an extensive theoretical study 
of the binary space spanned by the blocks
of a Steiner triple system.
A formula for the exact number of different STS$(2^m-1)$ of $2$-rank $2^m-m$,
which is one more than the minimum value, was derived by Tonchev \cite{Tonchev:2001:mass}.
V.\,Zinoviev and D.\,Zinoviev \cite{ZZ:2012:STS,ZZ:2013:rem}
proved the formulas for the number of different STS$(2^m-1)$ of rank $2^m-m+1$ 
 and for the number of such STS with prescribed linear span.
 Four years later, D.\,Zinoviev \cite{Zin:2016:NumSTS} found the corresponding formulas
 for the next value of the rank, $2^m-m+2$.
Recently, Jungnickel and Tonchev \cite{JunTon:18}
described the possible structures of the linear span of the blocks 
of a Steiner triple systems over $\mathrm{GF}(2)$ or $\mathrm{GF}(3)$,
in terms of parity check matrices.
In the subsequent paper \cite{JunTon:18+}, Jungnickel and Tonchev 
proved characterization theorems for Steiner triple systems 
whose linear span, binary or ternary, lies in a prescribed subspace; 
it occurs that such systems are in one-to-one correspondence with special collections of designs of smaller order 
(STS, transversal designs, $1$-factorizations), 
which yield a formula for the number of such systems, i.e., of different systems 
whose linear span is upperbounded by a prescribed subspace.
For the partial cases of STS$(2^m-1)$ of $2$-rank $2^m-m+1$ and $2^m-m+2$,
similar results were obtained earlier by V.\,Zinoviev and D.\,Zinoviev \cite{ZZ:2012:STS,ZZ:2013:struct}.
The goal of our paper is to make the next step for the general case 
and derive 
a formula for the number of all different STS$(v)$ whose rank coincides with a given value 
(less than $v$ for the $2$-rank  or less than $v-1$ for the $3$-rank),
but the linear span is not limited by any prescribed subspace.

Another very important direction in the enumeration of any kind of combinatorial
configurations is evaluating the number of nonisomorphic configurations,
and the Steiner triple systems of prescribed rank are not an exception.
Two STS on the same point set $S$ are \emph{isomorphic} if there is a permutation
of $S$ (an \emph{isomorphism}) that sends the blocks of one system to the blocks of the other.
Based on the exact formula on the number of different STS$(2^m-1)$ of $2$-rank $2^m-m$,
Tonchev \cite{Tonchev:2001:mass} derived an exponential lower bound on the number of nonisomorphic
systems with these parameters.
Recently, Jungnickel and Tonchev \cite{JunTon:18+} generalized that result
and obtained estimates
for the number of isomorphism classes of STS$(2^m w-1)$ of $2$-rank 
$2^m w - 1 -m$ and  STS$(3^m w)$ of $3$-rank $3^m w - m - 1$, for any positive integer $w$ and $m$
(formally, the statements are formulated for $w=2^t$ and $w=3^t$, respectively, 
but the theory developed works for an arbitrary $w$ as well).

There are several computational results related
with the calculation of the exact number 
of nonisomorphic Steiner triple systems of given $2$- or $3$-rank.
We summarize all known numbers in Table~\ref{t:r}.
There are $80$ nonisomorphic STS$(15)$; their $2$-ranks were studied 
by Tonchev and Weishaar \cite{TonWei:1997}.
Stinson and Seah \cite{StiSea:1986} calculated the number $284 457$ 
of isomorphism classes of STS$(19)$ 
that have a sub-STS of order $9$; these systems are exactly the  STS$(19)$ of $2$-rank $18$.
Kaski and {\"O}sterg{\aa}rd \cite{KO:STS19} classified all  nonisomorphic STS$(19)$;
their number is $11084874829$, and hence the number of STS$(19)$ of $2$-rank $19$ is also known.
Osuna \cite{Osuna:2006} found that there are $1239$ nonisomorphic STS$(31)$ of $2$-rank $27$.
Kaski, {\"O}sterg{\aa}rd,  Topalova,  and Zlatarski \cite{KOTZ:2008} counted the number $2166351$ 
of nonisomorphic STS$(21)$ of Wilson-type, which are essentially the STS$(21)$ of $3$-rank $19$.
Recently, Jungnickel, Magliveras, Tonchev, and Wassermann \cite{JMTW:STS27} calculated the number
of STS$(27)$ of $3$-rank $24$.

\begin{table}[t]
 \begin{tabular}{||@{\,}c@{\,}||@{\,}c@{\,}|@{\,}c@{\,}|@{\,}c@{\,}|@{\,}c@{\,}|@{\,}c@{\,}|@{\,}c@{\,}||@{\,}c@{\,}|@{\,}c@{\,}|@{\,}c@{\,}|@{\,}c@{\,}||}\hline
order & \multicolumn{10}{c||}{Number of nonisomorphic STS} \\ \hline
$v$   & \multicolumn{6}{c||}{of $2$-rank} & \multicolumn{4}{c||}{of $3$-rank } \\ \hline
     & $v$ & $v-1$ & $v-2$ & $v-3$ & $v-4$ & $v-5$ & $v-1$ & $v-2$ & $v-3$ & $v-4$ \\ \hline\hline
7 & 0 & 0 & 0 & 0 & 1                 & $-$ & 1 & $-$ & $-$ & $-$ \\ \hline
9 & 1 & $-$ & $-$ & $-$ & $-$                 & $-$ & 0 & 0 & 1  & $-$ \\ \hline
13 & 2  & $-$ & $-$ & $-$ & $-$               & $-$ & 2 & $-$ & $-$  & $-$ \\ \hline
15 & 57 & 16 & 5 & 1 & 1              & $-$ & 80 & $-$ & $-$  & $-$ \\ \hline
19 & 11084590372 & 284457 & $-$ & $-$ & $-$ & $-$ & 11084874829 & $-$ & $-$  & $-$ \\ \hline
21 & ? &  $-$ & $-$ & $-$ & $-$               & $-$ & ? & 2166351 & $-$  & $-$ \\ \hline
25 & ?  & $-$ & $-$ & $-$ & $-$               & $-$ & ? & $-$ & $-$  & $-$ \\ \hline
27 & ?  & ? & $-$ & $-$ & $-$               & $-$ & ? & ? & 2624  & 1 \\ \hline       
29 & ?  & $-$ & $-$ & $-$ & $-$               & $-$ & ? & $-$ & $-$  & $-$ \\ \hline
31 & ? &  ? & ? & ? & 1239 & 1 & ? & $-$ & $-$ & $-$                   \\ \hline
\end{tabular}
\caption{Known results on the number of isomorphism classes of STS of given rank}
\label{t:r}
\end{table}

In the current paper, we derive formulas for the number of STS$(v)$ of arbitrary $2$-rank smaller than $v$, see Theorem~\ref{th:N2},
or $3$-rank smaller than $v-1$ (note that the $3$-rank of STS cannot exceed $v-1$,
as it is always orthogonal to the all-one vector $(1,1,\ldots,1)$ over $\mathrm{GF}(3)$), see Theorem~\ref{th:N3}.
In particular, our result generalizes the formulas 
for $2$-rank $2^m-m$ \cite{Tonchev:2001:mass}, $2^m-m+1$ \cite{ZZ:2013:rem}, and $2^m-m+2$ \cite{Zin:2016:NumSTS}, obtained before.
The generalization is based on the M\"obius transform, which makes possible to derive a common formula for different ranks
and also to simplify some arguments.
The formulas are tight but conditional:
they involve the number of objects of smaller order
(Steiner triple systems, $1$-factorizations of complete graph, and latin squares).

For partial cases where these numbers are known, we obtain explicit values.
Namely, in addition to the results known before,
 we get the number of STS$(3^k)$ of $3$-rank $v-k$ and $v-k+1$,
 the number of STS$(7\cdot 3^k)$ of $3$-rank $7\cdot 3^k-k-1$,
 and the number of STS$(10\cdot 2^k-1)$ of $2$-rank $10\cdot 2^k-k-1$,
 for every $k$ (Corollaries~\ref{c:3concrete} and~\ref{c:2concrete}).
 In the other cases, our formulas can be combined
 with the asymptotic estimations of the number of
 Steiner triple systems \cite{Wilson:74:STS,LinLur:2013:STS},
 $1$-factorizations \cite{Cameron:parall,LinLur:2013:STS,Zin:2014:factor},
 and latin squares, see e.g. \cite[Theorems~17.2, 17.3]{vL-W:comb2001},
 or with some bounds on the number of different objects of the required order.
 For example, from the formulas derived in the current paper (Corollary~\ref{c:3concrete}), we know that the number
 of STS$(27)$ of $3$-rank at most $25$ is 
 $124363532158160295774288076917115534405632000000$,
which gives a lower bound on the number of all STS$(27)$
 (however, this bound is, hypothetically, 
 incredibly small against the number of STS$(27)$ of $3$-rank $26$).

We emphasize that the main difference of our result with the exact formulas of Jungnickel and Tonchev \cite{JunTon:18+} is that the formula of \cite{JunTon:18+} (concerning the $2$-rank or the $3$-rank)
counts the number of distinct STS contained a given linear subspace, 
while the rank of the counted STS is \emph{less than or equal} the dimension $d$ of the subspace.
Our main attempt is focused on excluding from this formula the systems of rank smaller than $d$
and counting only the STS whose linear span is \emph{exactly} the given subspace.
After that, multiplying by a simple factor, we obtain the total number of STS of rank $d$.
 
In the next section, we define necessary concepts and mention related facts.
In Section~\ref{s:3},
we describe the structure of an STS$(v)$ of prescribed $3$-rank $r_3$, $r_3<v-1$,
and of its dual space and derive a formula for the number of such systems.
In Section~\ref{s:2}, similar results are obtained for STS$(v)$ of given $2$-rank $r_2$, $r_2<v$.
In concluding Section~\ref{s:concl}, we show that there is no Steiner triple system 
which is both non-full-$2$-rank and non-full-$3$-rank, 
and briefly discuss the number of the isomorphism
classes of STS of prescribed rank.
\section{Definitions and notations}\label{s:preliminaries}

\paragraph{Orthogonality and duality.}\label{s:dual}
Two $v$-tuples $x=(x_1,\ldots,x_v)$ and $y=(y_1,\ldots,y_v)$
understood as vectors over GF$(q)$
are said to be \emph{orthogonal},
denoted $x\perp y$,
if $x_1y_1+\ldots + x_vy_v=0$.
Given a set $B$ of vectors, the \emph{dual space} $B^\perp$ is the set of all vectors orthogonal to each element of $B$.
By $\langle B \rangle$, we denote the linear span of the vector set $B$.

To simplify the formulas, we will use the standard notation of $q$-factorial $[n]_q!=\prod_{s=1}^n\sum_{i=0}^{s-1}q^i$.
Using this notation, the number $\prod_{t=1}^{n} (q^n-q^{t-1})$ of different bases in an $n$-dimensional space over GF$(q)$
can be written as $q^{\frac{n(n-1)}2}(q-1)^n[n]_q!$.

\paragraph{Latin squares.}\label{s:latin}
A \emph{latin square} of order $n$ is a function $f:\{1,\ldots,n\}^2\to \{1,\ldots,n\}$ invertible in each argument.
Traditionally, latin squares are represented by their value tables, whose rows and columns, by definition, contain all elements from $1$ to $n$.
(A system from the set $\{1,\ldots,n\}$ with a latin square operation $f$ is known as a \emph{quasigroup} of order $n$.)
A latin square $f$ is called \emph{symmetric} if $f(x,y)\equiv f(y,x)$ (i.e., the corresponding quasigroup is commutative).
A latin square $f$ is called \emph{totally symmetric} if $f(x,y)=z$ implies $f(y,x)=z$, $f(x,z)=y$, $f(z,x)=y$, $f(y,z)=x$, and $f(z,y)=x$.
A latin square $f$ is called \emph{idempotent} if $f(x,x)\equiv x$.
It is well known and obvious that the idempotent totally symmetric latin squares of order $n$ are in one-to-one correspondence with the Steiner triple systems of order $n$.
\begin{proposition}\label{p:STS-TSLS}
Let $S= \{1,\ldots,n\}$.
For every Steiner triple system $(S,B)$, the function $f$ defined as $f(x,x)\equiv x$ and $f(x,y)=z$ for every $\{x,y,z\}\in B$ is an idempotent totally symmetric latin square.
Conversely, every idempotent totally symmetric latin square $f:S^2\to S$ induces the  Steiner triple system $(S,B)$, $B=\{\{x,y,z\}:x\ne y, f(x,y)=z\}$.
\end{proposition}
Slightly less obvious but also well known is the following bijection.
\begin{proposition}\label{p:SLS-SLS0}
For every odd $n$, there is a one-to-one correspondence between symmetric latin squares of order $n$ and symmetric latin squares $f$ of order $n+1$ such that $f(x,x)\equiv n+1$.
\end{proposition}
\begin{proof}
Given symmetric latin square $f$ of order $n+1$ such that $f(x,x)\equiv n+1$,
the function $g:\{1,\ldots,n\}^2\to \{1,\ldots,n\}$ defined by
$g(x,x) = f(x,n+1)$, $g(x,y)=f(x,y)$ for every different $x$ and $y$ from $\{1,\ldots,n\}$ is straightforwardly a symmetric latin square.
$$
\begin{array}{|cccc|} \hline
 \mathbf 4 & 1 & 3 &  \mathbf 2\\
1 &  \mathbf 4 & 2 &  \mathbf 3\\
3 & 2 &  \mathbf 4 &  \mathbf 1\\
 \mathbf 2 &  \mathbf 3 &  \mathbf 1 &  \mathbf 4\\  \hline
\end{array}
\longleftrightarrow
\begin{array}{|ccc|} \hline
 \mathbf 2 & 1 & 3\\
1 &  \mathbf 3 & 2 \\
3 & 2 & \mathbf 1 \\ \hline
\end{array}
$$
To see the converse,
it is important to note that for a symmetric latin square $g$ of odd order $n$,
the set $\{g(x,x):x\in  \{1,\ldots,n\}\}$ coincides with  $\{1,\ldots,n\}$
(indeed, for every $x$ the number of the pairs $(y,z)$ such that $g(y,z)=x$ is $n$, which is odd,
while the number of the pairs $(y,z)$ such that $g(y,z)=x$ and $y\ne z$ is even, from the symmetry).
Then we define the required $f$ by the identities $f(x,x)= n+1$, $f(x,n+1)= f(n+1,x)= g(x,x)$,
and $f(x,y)=g(x,y)$ for every $x,y\in \{1,\ldots,n\}$, $x\ne y$.
\end{proof}
\begin{remark}
The symmetric latin squares $f$ of even order $n$ such that $f(x,x)\equiv n$
are in a straightforward one-to-one correspondence with the ordered $1$-factorizations of the complete graph on the vertex set $\{1,\ldots,n\}$
(i.e., the ordered partitions of the set of edges of this graph into $n-1$ sets of mutually disjoint edges; the number of the ordered partitions equals $(n-1)!$ times the number of unordered partitions, see~\cite{A000438}),
with the tournament schedules for $n$ teams, see e.g. \cite{A036981},
and with the resolutions of the complete system of pairs from $\{1,\ldots,n\}$, see e.g. \cite{Assmus:95}.
Under these different names, but in the same context as in the current paper,
the symmetric latin squares can be mentioned in the literature on combinatorial designs.
\end{remark}

\paragraph{M\"obius coefficients.}\label{s:Mobius}
 For a prime power $q$,
define the numbers $\mu^{\scriptscriptstyle(q)}_{i}$, $i=0,1,2,\ldots$ by the following recursion:
$\mu^{\scriptscriptstyle(q)}_{0}=1$, and for an $i$-dimensional space $S$ over GF$(q)$, $i\ge 1$, and the set $\mathcal S$
of its subspaces,
\begin{equation}\label{eq:mob}
   \mu^{\scriptscriptstyle(q)}_{i}=-\sum_{C \in \mathcal S \backslash\{S\}} \mu^{\scriptscriptstyle(q)}_{\mathrm{dim}(C)},
    \qquad \mbox{or, equivalently,}\qquad
    \sum_{C \in \mathcal S} \mu^{\scriptscriptstyle(q)}_{\mathrm{dim}(C)}=0.
\end{equation}

\begin{remark}
The numbers $\mu^{\scriptscriptstyle(q)}_{i}$ are related with the so-called M\"obius function on the poset of spaces over GF$(q)$.
Namely, for two spaces $U$ and $V$,
the M\"obius function $\mu_U(V)$ equals $\mu^{\scriptscriptstyle(q)}_{\mathrm{dim}(U)-\mathrm{dim}(V)}$
if $V\subseteq U$ and $0$ otherwise.
\end{remark}

\begin{lemma}[{\cite[Sect.\,3.10]{Stanley}}]\label{l:Mob}
For every prime power $q$, one has
$\mu^{\scriptscriptstyle(q)}_{i}=(-1)^iq^{\binom{i}{2}}$.
\end{lemma}

\section{Non-full-3-rank STS}\label{s:3}
Let $v\equiv 1,3 \bmod 6$; that is, there exist STS$(v)$.
By $V^v$, we denote the vector space of all $v$-tuples over $\mathrm{GF}(3)$.
Denote by $\mathcal{D}$ the set of subspaces of $V^v$,
each including the all-one vector and being orthogonal to at least one STS$(v)$;
denote $$\mathcal{D}_i = \{D\in\mathcal{D} : \dim(D)=i+1 \}.$$
The following lemma can be considered as a treatment of the results of \cite[Sect.~4]{DHV:1978}
in terms of the structure of a basis for the dual space of STS.

\begin{lemma}[{\cite[Thm 5.1]{JunTon:18}}]\label{l:3i}
 Let $M$ be an $(i+1)\times v$ generator matrix for $D \in \mathcal{D}_i$,
 and let the first row of $M$ be the all-one vector.
 Then $M$ consists of $3^i$ different columns, each occurring $v/3^i$ times; e.g.,
$$
\left(\begin{array}{c@{\,}c@{\,}c@{\ }c@{\,}c@{\,}c@{\ }c@{\,}c@{\,}c@{\ }c@{\,}c@{\,}c@{\ }c@{\,}c@{\,}c@{\ }c@{\,}c@{\,}c@{\ }c@{\,}c@{\,}c@{\ }c@{\,}c@{\,}c@{\ }c@{\,}c@{\,}c@{\ }c@{\,}c@{\,}c}
     1&1&1&1&1&1&1&1&1&1&1&1&1&1&1&1&1&1&1&1&1&1&1&1&1&1&1  \\
     0&0&0&0&0&0&0&0&0&1&1&1&1&1&1&1&1&1&2&2&2&2&2&2&2&2&2  \\
     0&0&0&1&1&1&2&2&2&0&0&0&1&1&1&2&2&2&0&0&0&1&1&1&2&2&2
\end{array}\right).
$$
\end{lemma}
The proof given in \cite{JunTon:18} includes a proof of more deep mathematical result, a variation of Bonisoli's theorem. We give an independent simple proof.
\begin{proof}
Without loss of generality,
we can assume that the first column of $M$ is $(1,0,\ldots,0)^{\mathrm T}$
(we can achieve this by subtracting the first row from some of the others).

Claim (*). \emph{If $a$ and $b$ are columns of $M$, then $-a-b$ is also a column of $M$}.
Consider an STS$(v)$ orthogonal to the rows of $M$.
Let $a$ and $b$ be the $j^\text{th}$ and $k^\text{th}$ columns of $M$, and let $\{j,k,l\}$ be the STS block  containing $j$ and $k$.
Since all rows of $M$ are orthogonal to the characteristic vector of this block, the $l^\text{th}$ column $c$ satisfies $a+b+c=0$, i.e., $c=-a-b$. This proves (*).

Claim (**). \emph{If $c$ and $d$ are columns of $M$, then $c+d-(1,0,\ldots,0)^{\mathrm T}$ is also a column of $M$}. This is proved by applying (*) with $a=c$, $b=d$ first, and then with $a=-c-d$, $b=(1,0,\ldots,0)^{\mathrm T}$.

The last claim means that the set of columns of the matrix $M'$ obtained from $M$ by removing the first row is closed under addition. Since there are $i$ linearly independent columns, this set contains all possible $3^i$ columns of height $i$.

It remains to prove that different columns $a$ and $b$ occur the same number of times in $M$. Let $J$ and $K$ be the sets of positions in which $M$ has the columns $a$ and $b$, respectively. And let $l$ be a position of the column $-a-b$. For each $j$ from $J$, there is $k$ from $K$ such that $\{j,k,l\}$ is a block of the STS. Moreover, different $j$s correspond to different $k$s. This shows that $|J|\le |K|$. Similarly,  $|K|\le |J|$.
\end{proof}

\begin{lemma}\label{l:Gamma}
Let $i$ and $j$ be nonnegative integers such that $i\le j$.
If $\mathcal{D}_{j}$ is not empty, then every subspace from $\mathcal{D}_i$
is contained in exactly $\Gamma_{v,i,j}$ subspaces from $\mathcal{D}_j$, where
\begin{equation}\label{eq:Gamma}
\Gamma_{v,i,j} =
 \left(\frac v{3^i} !\right)^{\!\!3^i} \Big/
3^{\frac{(j+i+1)(j-i)}2}
\left(\frac v{3^j} !\right)^{\!\!3^{j}}
 2^{j-i}[j-i]_3!.
\end{equation}
In particular, $|\mathcal{D}_j|= \Gamma_{v,0,j}$.
\end{lemma}
\begin{proof}
 Firstly, let us fix some $D_i$ from $\mathcal{D}_i$ and construct all $D_{j}$
 from $\mathcal{D}_j$ that satisfy $D_i\subseteq D_{j}$. 
         Let ${M_i}$ be a generator matrix of $D_i$ whose first row contains only $1$s.
         According to Lemma~\ref{l:3i},  ${M_i}$ divides the coordinates into $3^i$ ``cells''
          such that each cell contains the same columns in it.
          Since $D_i\subseteq D_{j}$, so $D_{j}$ has a generator matrix ${M_{j}}$
          that starts with the $i+1$ rows of ${M_i}$,
          and this matrix subdivides the cells into $3^{j}$ ``subcells'' of the same size.
          The number of such subdivisions is
          $A={\left(\frac{v}{3^i}!\right)^{3^i}} \big/ {\left(\frac{v}{3^{j}}!\right)^{3^{j}}}.$

         Next, let us count the number of such matrices that generate the same code.
         Let ${M'_{j}}$ be another generator matrix of $D_{j}$ that also starts with $i+1$ rows of ${M_i}$;
         what is more, its $t^{\text{th}}$ row,  $i+2 \le t \le {j}+1$,
         is a linear combination of the rows of ${M_{j}}$,
         but is not a linear combination of the rows above
         it in the matrix ${M'_{j}}$.
         So, we can get there are $3^{{j}+1}-3^{i+1}$ kinds of values for the row $i+2$,
         $3^{{j}+1}-3^{i+2}$ kinds of values for the row $i+3$ and so on.
         Therefore, the number of such matrices $M'_{j}$ is $B=(3^{{j}+1}-3^{i+1})(3^{{j}+1}-3^{i+2})\ldots(3^{{j}+1}-3^{j})=3^{\frac{({j}-i)({j}+i+1)}{2}}\prod\limits_{s=1}^{{j}-i}(3^s-1).$
         Finally, 
         the number of different $D_{j}$ that satisfy $D_{i}\subseteq D_j$ is $\Gamma_{v,i,j}=\frac{A}{B}$.
\end{proof}

\begin{lemma}[the structure of non-full-$3$-rank STS\cite{JunTon:18+}]\label{l:STS3Lat}
Given a subspace $D$ from $\mathcal{D}_j$,
the set of STS$(v)$ orthogonal to $D$ is in one-to-one correspondence
with the collections of $3^j$ Steiner triple systems of order $v/3^j$ and ${3^j(3^j-1)/6}$ latin squares of order $v/3^j$.
\end{lemma}
\proof[Proof (a sketch)]
According to Lemma~\ref{l:3i}, a generator matrix $M$ of $D$ divides the coordinates
into $3^j$ \emph{groups} of size $v/3^j$. 
It can be seen that every STS$(v)$ orthogonal to $D$
is divided into the $3^j+{3^j(3^j-1)/6}$ following subsets:
\begin{itemize}
    \item For each of $3^j$ groups, 
    the triples with all $3$ points in these group form STS$(v/3^j)$.
    \item For every $3$ distinct groups $\{\alpha_1,\ldots,\alpha_{v/3^j}\}$,
     $\{\beta_1,\ldots,\beta_{v/3^j}\}$,  $\{\gamma_1,\ldots,\gamma_{v/3^j}\}$ corresponding to columns $a$, $b$, $c$ 
    with $a+b+c=0$,  
    the  triples with one point in each of these $3$ groups have the form $\{\alpha_x,\beta_y,\gamma_{f(x,y)}\}$ for some latin square $f$ of order~$v/3^j$. \qed
\end{itemize}

\begin{corollary}[{\cite{JunTon:18+}}]\label{c:STS3Lat}
Given a subspace $D$ from $\mathcal{D}_j$,
the number $\Phi(D)$ of STS$(v)$ orthogonal to $D$ equals $\Phi_{v,j}$,
where
$$
\Phi_{v,j}=\Psi_{v/3^j}^{3^j}
\Lambda_{v/3^j}^{3^j(3^j-1)/6},
$$
$\Psi_{u}$ is the number of STS$(u)$,
and $\Lambda_u$ is the number of latin squares of order $u$.
\end{corollary}

Now, given a subspace $D$ from $\mathcal{D}_j$, we know the number $\Phi(D)$ of STS that are orthogonal to some subspace of $D$.
To find the number of STS that are dual to $D$,
we should apply to the function $\Phi(D)$ the M\"obius transform on the poset of subspaces of $D$.
This is essentially done in the next lemma.

\begin{lemma}\label{l:ups}
Assume that $v$ is divided by $3^k$ and $k$
is the largest integer with this property.
Let $i \in \{0,\ldots, k\}$,
and let $D$ be in $\mathcal{D}_i$.
The number of STS$(v)$ with dual space $D$ equals $\Upsilon_{v,i}$, where
\begin{equation}\label{eq:ups}
 \Upsilon_{v,i} = \sum_{j=i}^k \Gamma_{v,i,j} \mu^{\scriptscriptstyle(3)}_{j-i}  \Phi_{v,j},
\end{equation}
where $\Gamma_{v,i,j}$ and $\Phi_{v,j}$
are from Lemma~\ref{l:Gamma} and Corollary~\ref{c:STS3Lat}.
\end{lemma}
\begin{proof}
Utilizing the definition of $\Gamma_{v,i,j}$ and then expanding $\Phi_{v,j}$, we have
$$
\sum_{j=i}^k \Gamma_{v,i,j} \mu^{\scriptscriptstyle(3)}_{j-i}  \Phi_{v,j}
= \sum_{j=i}^k \sum_{\genfrac{}{}{0pt}{}{D'\in\mathcal{D}_j}{D\subseteq D'}} \mu^{\scriptscriptstyle(3)}_{j-i}  \Phi_{v,j}
= \sum_{j=i}^k \sum_{\genfrac{}{}{0pt}{}{D'\in\mathcal{D}_j}{D\subseteq D'}} \sum_{B\in P(D')} \mu^{\scriptscriptstyle(3)}_{j-i},$$
where $P(D')$ is the set of STS$(v)$ orthogonal to $D'$.
We continue:
\begin{eqnarray*}
&&
= \sum_{\genfrac{}{}{0pt}{}{D'\in\mathcal{D}}{D\subseteq D'}} \sum_{B\in P(D')} \mu^{\scriptscriptstyle(3)}_{\dim(D')-1-i}
= \sum_{\genfrac{}{}{0pt}{}{D'\in\mathcal{D}}{D\subseteq D'}} \sum_{\genfrac{}{}{0pt}{}{B\in P(D)}{B\perp D'}} \mu^{\scriptscriptstyle(3)}_{\dim(D')-1-i}
\\ &&
= \sum_{B\in P(D)} \sum_{\genfrac{}{}{0pt}{}{D'\in\mathcal{D}}{D\subseteq D',\,B\perp D'}} \mu^{\scriptscriptstyle(3)}_{\dim(D')-1-i}
=\sum_{B\in P(D)} \sum_{\genfrac{}{}{0pt}{}{D'\in \mathcal D:}{D\subseteq D'\subseteq B^\perp}}
 \mu^{\scriptscriptstyle(3)}_{\dim(D')-1-i}
 \\ &&
 \stackrel{D^*=D'/D }{=} 
\sum_{B\in P(D)}
\sum_{D^*\subseteq B^\perp/D}
 \mu^{\scriptscriptstyle(3)}_{\dim(D^*)}
 \stackrel{(\ref{eq:mob})}{=}
\sum_{B\in P(D)}
 (\mbox{$1$ if $B^\perp = D$; $0$ otherwise}).
\end{eqnarray*}
We see that the last formula meets the definition of $\Upsilon_{v,i}$.
\end{proof}
\begin{theorem}\label{th:N3}
Assume that $v$ is divided by $3^k$ and $k$
is the largest integer with this property.
Let $i \in \{0,\ldots, k\}$.
The total number of different STS$(v)$ of $3$-rank $v-i-1$ equals
$$ 
 \Gamma_{v,0,i} \sum_{j=i}^k \Gamma_{v,i,j} \mu^{\scriptscriptstyle(3)}_{j-i} \Phi_{v,j},
 \qquad
 \mbox{ where }
 \mu^{\scriptscriptstyle(3)}_{l}=(-1)^l3^{\binom{l}{2}},
 \quad
\Phi_{v,j}=\Psi_{v/3^j}^{3^j}
\Lambda_{v/3^j}^{3^j(3^j-1)/6},
$$
$$
\Gamma_{v,i,j} =
 \left(\frac v{3^i} !\right)^{\!\!3^i} \bigg/
3^{\frac{(j+i+1)(j-i)}2}
\left(\frac v{3^j} !\right)^{\!\!3^{j}}
 \prod_{s=1}^{j-i}(3^s-1),
$$
$\Psi_{u}$ is the number of STS$(u)$
(and also the number of idempotent totally symmetric latin squares of order $u$),
and $\Lambda_u$ is the number of latin squares of order $u$.
\end{theorem}
\begin{proof}
The  number of STS$(v)$ of $3$-rank $v-i-1$ equals the number $\Upsilon_{v,i}$ of STS$(v)$ of $3$-rank $v-i-1$ orthogonal to a given subspace $D$ from $\mathcal{D}_i$
multiplied by the number $\Gamma_{v,0,i}$ of such subspaces. Utilizing the formulas from Lemma~\ref{l:ups} and Corollary~\ref{c:STS3Lat}, we get the result.
\end{proof}
\begin{corollary}\label{c:3concrete}
The number of STS$(v)$, $v=3^k$, of $3$-rank $v-k-1$ is
$$
\frac{v!}{3^{\frac{k(k+1)}2}\cdot 2^k \cdot[k]_3!}.
$$
The number of STS$(v)$, $v=3^{k}$, of $3$-rank $v-k$ is
$$
\frac{v!\cdot(2^{v^2/27-4v/9+1}
\cdot 3^{{v^2/54-7v/18}+k}-1)}
{2^k\cdot3^{\frac{k(k+1)}2}\cdot [k-1]_3!}. 
$$
The number of STS$(v)$, $v=3^k$, of $3$-rank $v-k+1$ is
$$
\frac{v!}{2^{k+2} \cdot 3^{\frac{k(k+1)}2-1} \cdot [k-2]_3!} 
$$
$$ \times\left(\frac{(2^{35}\cdot 3^8\cdot 5^2\cdot 7^2\cdot 5231 \cdot 3824477)^{\frac{v(v-9)}{486}}}{2^{4v/9-4}\cdot3^{v/3-2k+1}}
-
2^{v^2/27-4v/9+3}\cdot 3^{v^2/54-7v/18+k-1} + 1 \right).
$$
The number of STS$(v)$, $v=7\cdot 3^k$, of $3$-rank $v-k-1$ is
$$
\frac{v! \cdot 61479419904000^{\frac{3^k(3^k-1)}6}}
{2^k\cdot 3^{\frac{k(k+1)}2}\cdot 168^{\,3^k}\cdot [k]_3!}. 
$$
\end{corollary}
\begin{proof}
According to \cite{A002860}, we have
$\Lambda_1=1$, $\Lambda_3=12$,  $\Lambda_7=61479419904000 =2^{18} \cdot 3^5 \cdot 5^3 \cdot 7 \cdot 1103$
\cite{Sade:48,Saxena:7x7},
$\Lambda_9=5524751496156892842531225600= 2^{35}\cdot 3^8\cdot 5^2\cdot 7^2\cdot 5231 \cdot 3824477$ \cite{BamRoth:75}
(the last known value is $\Lambda_{11}$ \cite{MK-W:11}).
According to \cite{A030128},
we have $\Psi_1=\Psi_3=1$, $\Psi_7=30$, $\Psi_9=840$
(the last known value is $\Psi_{19}$ \cite{KO:STS19}).
Applying the result of Theorem~\ref{th:N3}, we get the formulas.
\end{proof}
A computer-aided classification of equivalence classes
of STS$(27)$ of $3$-rank $24$ is described in \cite{JMTW:STS27}. 
In particular, the total number of different systems with these parameters can be calculated from \cite[Table 1]{JMTW:STS27} as the sum $\sum N_r27!/|\mathrm{Aut}(\mathcal{S})|$ over the all rows of the table except the last one (corresponding to the $3$-rank $23$). This number coincides with the one given by our formula, $22\,300\,404\,167\,684\,260\,773\,163\,008\,000\,000$.
\section{Non-full-2-rank STS}\label{s:2}

In this section, to simplify the formulas, we denote the order of STS by $w-1$ instead of $v$.
By $\dot V^{w-1}$, we denote the vector space of all $(w-1)$-tuples over $\mathrm{GF}(2)$.
Denote by $\dot{\mathcal{D}}_i$ the set of $i$-dimensional subspaces of $\dot V^{w-1}$
orthogonal to at least one STS$(w-1)$.
The following lemma can be considered as a treatment of the results of \cite[Sect.~3]{DHV:1978}
in terms of the structure of a basis for the dual space of STS.

\begin{lemma}[{\cite[Thm 4.1]{JunTon:18}}]\label{l:2i}
 Let $M$ be an $i\times (w-1)$ generator matrix for $D \in \dot{\mathcal{D}}_i$.
 Then each of the $2^i-1$ nonzero columns of height $i$
 occurs $w/2^i$ times as a column of $M$,
 while the all-zero column occurs $w/2^i-1$ times; e.g.,
$$\left(\begin{array}{c@{\,}c@{\,}c@{\ }c@{\,}c@{\,}c@{\,}c@{\ }c@{\,}c@{\,}c@{\,}c@{\ }c@{\,}c@{\,}c@{\,}c@{\ }c@{\,}c@{\,}c@{\,}c@{\ }c@{\,}c@{\,}c@{\,}c@{\ }c@{\,}c@{\,}c@{\,}c@{\ }c@{\,}c@{\,}c@{\,}c}
0&0&0&0&0&0&0&0&0&0&0&0&0&0&0& 1&1&1&1&1&1&1&1&1&1&1&1&1&1&1&1 \\
0&0&0&0&0&0&0&1&1&1&1&1&1&1&1& 0&0&0&0&0&0&0&0&1&1&1&1&1&1&1&1 \\
0&0&0&1&1&1&1&0&0&0&0&1&1&1&1& 0&0&0&0&1&1&1&1&0&0&0&0&1&1&1&1
\end{array}\right).
$$
\end{lemma}
\begin{proof}
Claim (*). \emph{If $a$ and $b$ are different nonzero columns of $M$, then $a+b$ is also a column of $M$}.
The proof is similar to that of Claim (*) in the proof of Lemma~\ref{l:2i}.

Since the rank of $M$ is $i$, it contains $i$ linearly independent columns.
It follows from (*) that it contains all $2^i-1$ different nonzero columns of height $i$.

It remains to show that every nonzero column $a$ occurs $|K|+1$ times,
where $K$ is the set of positions in which $M$ has the all-zero column.
Let $J$ be the sets of positions in which $M$ has the column $a$, and let $l\in J$.
For each $j$ from $J\backslash \{l\}$, there is $k$ from $K$ such that $\{j,k,l\}$ is a block of the STS.
Moreover, different $j$s correspond to different $k$s. This shows that $|J\backslash \{l\}| \le |K|$. Similarly,  $|K|\le |J\backslash \{l\}|$.
\end{proof}

\begin{lemma}\label{l:Gamma2}
Let $i$ and $j$ be nonnegative integers such that $i\le j$.
If $\dot{\mathcal{D}}_{j}$ is not empty, then every subspace from $\dot{\mathcal{D}}_i$
is contained in exactly $\dot\Gamma_{w,i,j}$ subspaces from $\dot{\mathcal{D}}_j$, where
\begin{equation}\label{eq:Gamma2}
\dot\Gamma_{w,i,j} =
\left(\frac{w}{2^i}!\right)^{\!\!2^i} \bigg/
2^{\frac{(j-i)(j+i+1)}2}
\left(\frac{w}{2^j}!\right)^{\!\!2^j}
[j-i]_2!.
\end{equation}
In particular, $|\dot{\mathcal{D}}_j|= \dot\Gamma_{w,0,j}$.
\end{lemma}
\begin{proof}
The proof is similar to that of Lemma~\ref{l:Gamma}.
The difference is that the size of one cell, corresponding to the all-zero columns
of the generator matrix, is one less than for each of the other cells;
the same can be said for the subcells. So, totally we have
$$
A={\left(\frac{w}{2^i}!\right)^{\!\!2^i-1}\left(\frac{w}{2^i}-1\right)!}\Big/{\left(\frac{w}{2^j}!\right)^{\!\!2^j-1}\left(\frac{w}{2^j}-1\right)!}
 ={\left(\frac{w}{2^i}!\right)^{\!\!2^i}}\Big/{2^{j-i}\left(\frac{w}{2^j}!\right)^{\!\!2^j}}$$
subdivisions. Dividing this number by the number
$B=2^{\frac{(j-i)(j+i-1)}{2}}\prod\limits_{s=1}^{j-i}(2^s-1)$
of the matrices generating the same space, we get the result.
\end{proof}

The following lemma describes the structure of an arbitrary non-full-$2$-rank STS.
It was proved in \cite{ZZ:2013:struct} for the partial case
of STS$(2^k-1)$ of rank $2^k+2$;
the arguments, however, are applicable to the general case.
It should be also noted that Theorem~4.1 in \cite{Assmus:95} is close to this result,
but the structure of the part of the block set connected with latin squares is not described there
(with the exception of one partial example in Remark~6).

\begin{lemma}[the structure of non-full-$2$-rank STS \cite{JunTon:18+}]\label{l:STS2Lat}
Given a subspace $D$ from $\dot{\mathcal{D}}_j$,
the set of STS$(w-1)$ orthogonal to $D$ is in one-to-one correspondence
with the collections of one STS$(w/2^j-1)$,
$2^j-1$ symmetric latin squares of order $w/2^j-1$,
and ${(2^j-1)(2^j-2)/6}$ latin squares of order $w/2^j$.
\end{lemma}
\proof[Proof (a sketch)]
According to Lemma~\ref{l:3i}, a generator matrix $M$ of $D$ divides the coordinates
into $2^j-1$ groups of size $w/2^j$ and one group
of size $w/2^j-1$ (the last group corresponds to the all-zero columns of $M$). 
It can be seen that the set of triples of every STS$(v)$ orthogonal to $D$
is divided into the $2^j+{(2^j-1)(2^j-2)/6}$ following subsets:
\begin{itemize}
    \item The triples with all $3$ points in the group of
    size $w/2^j-1$ form STS$(w/2^j-1)$.
    \item The triples with one points in the group 
    $\{\gamma_1,\ldots,\gamma_{w/2^j-1}\}$ of
    size $w/2^j-1$ 
    and two points in one of the $2^j-1$ groups 
    $\{\alpha_1,\ldots,\alpha_{w/2^j}\}$ of size $w/2^j$.
    Such triples have the form $\{\alpha_x,\alpha_y,\gamma_{f(x,y)}\}$ for some
    symmetric latin square $f$ satisfying $f(x,x)\equiv w/2^j$. Proposition~\ref{p:SLS-SLS0} relates $f$ with a symmetric latin square of order $w/2^j-1$.
    \item For every $3$ distinct groups $\{\alpha_1,\ldots,\alpha_{w/2^j}\}$,
     $\{\beta_1,\ldots,\beta_{w/2^j}\}$,  $\{\gamma_1,\ldots,\gamma_{w/2^j}\}$ corresponding to columns $a$, $b$, $c$ 
    with $a+b+c=0$,  
    the triples with one point in each of these $3$ groups have the form $\{\alpha_x,\beta_y,\gamma_{f(x,y)}\}$ for some latin square $f$ of order~$v/2^j$.
    \qed
\end{itemize}

\begin{corollary}[\cite{JunTon:18+}]\label{c:STS2Lat}
Given a subspace $D$ from $\dot{\mathcal{D}}_j$,
the number $\dot\Phi(D)$ of STS$(w-1)$ orthogonal to $D$ equals $\dot\Phi_{w-1,j}$,
where
$$
\dot\Phi_{w-1,j}=
\Psi_{w/2^j-1}
\Pi_{w/2^j-1}^{2^j-1}
\Lambda_{w/2^j}^{(2^j-1)(2^j-2)/6}
,$$
$\Psi_{u}$ is the number of STS$(u)$
(and the number of idempotent totally symmetric latin squares of order $u$),
$\Pi_u$ is the number of symmetric latin squares of order $u$,
$\Lambda_u$ is the number of latin squares of order $u$.
\end{corollary}

\begin{lemma}\label{l:ups2}
Assume that $w$ is divided by $2^k$ and $k$
is the largest integer with this property.
Let $i \in \{0,\ldots, k\}$,
and let $D$ be in $\dot{\mathcal{D}}_i$.
The number of STS$(w-1)$ with dual space $D$ equals $\dot\Upsilon_{w-1,i}$, where
\begin{equation}\label{eq:ups2}
 \dot\Upsilon_{w-1,i} = \sum_{j=i}^k \dot\Gamma_{w,i,j} \mu^{\scriptscriptstyle(2)}_{j-i}  \dot\Phi_{w-1,j},
\end{equation}
where $\dot\Gamma_{w,i,j}$ and $\dot\Phi_{w-1,j}$
are from Lemma~\ref{l:Gamma} and Corollary~\ref{c:STS3Lat}.
\end{lemma}
The proofs of the lemma and the following theorem are the same as those of Lemma~\ref{l:ups},
and we omit them.
\begin{theorem}\label{th:N2}
Assume that $w$ is divided by $2^k$ and $k$
is the largest integer with this property.
Let $i \in \{0,\ldots, k\}$.
The total number of different STS$(w-1)$ of $2$-rank $w-i-1$ equals
$$ 
 \dot\Gamma_{w,0,i} \sum_{j=i}^k \dot\Gamma_{w,i,j} \mu^{\scriptscriptstyle(2)}_{j-i} \dot\Phi_{w-1,j},
$$ 
where $\mu^{\scriptscriptstyle(2)}_{l}=(-1)^l2^{\binom{l}{2}}$, \
$
\dot\Phi_{w-1,j}=
\Psi_{w/2^j-1}
\Pi_{w/2^j-1}^{2^j-1}
\Lambda_{w/2^j}^{(2^j-1)(2^j-2)/6},
$ \
$\Psi_{u}$ is the number of STS$(u)$ (and also the number of idempotent totally symmetric latin squares of order $u$),\
$\Pi_{u}$ is the number of symmetric latin squares of order $u$ (and also $u!$ times the number of $1$-factorizations of the complete graph of order $u+1$),
$\Lambda_u$ is the number of latin squares of order $u$, and
$$
\dot\Gamma_{w,i,j} =
\left(\frac{w}{2^i}!\right)^{\!\!2^i} \bigg/
2^{\frac{(j-i)(j+i+1)}2}
\left(\frac{w}{2^j}!\right)^{\!\!2^j}
\prod_{s=1}^{j-i}(2^s-1).
 $$
\end{theorem}
\begin{corollary}\label{c:2concrete}
The number of STS$(w-1)$, $w=2^k$, of $2$-rank $w-k$ is
$$
w!
(2^{{w^2}/{24}-{3w}/{4}+k+1/3} - 1)
\big/
{2^{\frac{k(k+1)}2}[k-1]_2!}
\qquad\mbox{(see \cite{Tonchev:2001:mass})}.
$$
The number of STS$(w-1)$, $w=2^k$, of $2$-rank $w-k+1$ is
$$
\frac{w!
\left(
3^{w^2/48-w/4+2/3}\cdot 2^{w^2/16-5w/4+2k-1} -
3\cdot 2^{ w^2/24-3w/4+k-2/3}
+1
\right)
}{3\cdot 2^{\frac{(k+2)(k-1)}2}\cdot [k-2]_2!}
\qquad\mbox{(see \cite{ZZ:2013:rem})}.
$$
The number of STS$(w-1)$, $w=2^k$, of $2$-rank $w-k+2$ is
$$
\frac{2^k!}{21\cdot 2^{\frac{k(k+1)}2-3 }\cdot [k-3]_2!}$$
$${}\times
       \Big({780^{w/8-1}\cdot(2^{28} \cdot 3^5 \cdot 5^2 \cdot 7^2 \cdot 1361291)^{w^2/384-w/16+1/3} \cdot 2^{3k-12}} \qquad\qquad{}
       $$
       $$
       {}\qquad\qquad\qquad{}- 7\cdot 2^{w^2/16-5w/4+2k-3}
          \cdot 3^{w^2/48-w/4+2/3}
       + 7\cdot2^{w^2/24 -3w/4 - 5/3+k}
       -1 \Big)
\qquad\mbox{(see \cite{Zin:2016:NumSTS})}.
$$
The number of STS$(10w-1)$, $w=2^k$, of $2$-rank $10w-1-k$ is
$$
 \frac{(10w)!
 \cdot 122556672^{w-1}\cdot(2^{43} \cdot 3^{10} \cdot 5^4 \cdot 7^2 \cdot 31 \cdot 37 \cdot 547135293937)^{\frac{(w-1)(w-2)}6}
 }{2^{\frac{k(k+1)}2+5} \cdot 135 \cdot [k]_2!}.
$$
\end{corollary}
\begin{proof}
To apply the formula from Theorem~\ref{th:N2}, in addition to the values considered in the proof of Corollary~\ref{c:3concrete},
we need $\Pi_1=1$, $\Pi_3=6$, $\Pi_7=31449600=7!\cdot 6240$ \cite{WalStrWal72}, $\Pi_9=444733651353600=9!\cdot 1225566720$ \cite{GelOd:74}, see also \cite{A036981},
\begin{multline*}
\Lambda_{2}=2,\qquad
\Lambda_{4}= 576,\qquad
 \Lambda_{8}=108776032459082956800=2^{28} \cdot 3^5 \cdot 5^2 \cdot 7^2 \cdot 1361291 \quad\mbox{\cite{Wells:67}},
\\
\Lambda_{10}=9982437658213039871725064756920320000=
2^{43} \cdot 3^{10} \cdot 5^4 \cdot 7^2 \cdot 31 \cdot 37 \cdot 547135293937\quad \mbox{\cite{MK-Rog:10}},
\end{multline*}
see also \cite{A002860}.
We also formally need the trivial values $\Pi_0=1$ and $\Phi_0=1$.
\end{proof}

\begin{remark}
Taking into account Propositions~\ref{p:STS-TSLS} and ~\ref{p:SLS-SLS0}, we know that
$\Psi(u-1)$ and $\Pi(u-1)$ are the numbers of latin squares of order $u$ with certain restrictions.
So, $\Psi(u-1)<\Pi(u-1)<\Lambda(u)$. It can be then noted that if $j\ge 3$,
then the most valuable factor in the formulas for the number of $STS(v)$ of $2$-
or $3$-rank at most $v-j$ is connected with the number of unrestricted latin squares.
\end{remark}

\section{Concluding remarks}\label{s:concl}

As we see from 
the results of \cite{JunTon:18+},
the structure of the Steiner triple systems of deficient rank,
either $2$- or $3$-rank, with fixed orthogonal subspace, is well understood, meaning that it can be described
in terms of latin squares and 
Steiner triple systems of smaller order.
The possibility to derive an explicit formula 
for the number of the non-full-rank STS
(involving the number of latin squares and smaller STS)
implies that this description is constructive even
if we do not fix the orthogonal subspace of the systems.
The following simple statement shows 
that the benefits given by the knowledge of the structure
of a Steiner triple system depending on the value of its $2$-rank 
or the value of its $3$-rank cannot be combined in the same system.

\begin{theorem}\label{th:0}
There is no a Steiner triple system  of order $v$ larger than $3$ that is both non-full-$2$-rank and non-full-$3$-rank,
i.e., of $2$-rank less than $v$ and $3$-rank less than $v-1$.
\end{theorem}
\begin{proof}
Assume that $S$ an STS$(v)$, $v>3$, which is (i) of $3$-rank at most $v-2$ and (ii) $2$-rank at most $v-1$.
By Lemma~\ref{l:3i}, (i) means that there is a vector with $v/3$ zeros, $v/3$ ones, and $v/3$ twos that is orthogonal to $S$ over GF$(3)$;
in particular, $v\equiv 0\bmod 3$.
Assumption (ii) means that $S$ has a Steiner subsystem $S'$ of order $(v-1)/2$, by Lemma~\ref{l:STS2Lat}.
Since  $v>3$ implies $(v-1)/2>v/3$, the system $S'$ is orthogonal over GF$(3)$ to a vector that is not all-$0$, all-$1$, or all-$2$.
By Lemma~\ref{l:3i}, the order $(v-1)/2$ is an integer divisible by $3$, and we get $v\equiv 1\bmod 3$, a contradiction.
\end{proof}

Finally, we briefly discuss the number of isomorphism classes of STS of a prescribed $2$- or $3$-rank,
which can be evaluated utilizing the following observation.
\begin{proposition}[{see \cite[eq. (12)]{JunTon:18+}, where $C$ is the whole space}]\label{p:iso}
The number $N$ of isomorphism classes of STS$(v)$ 
from some family $ \mathcal S$ closed under isomorphism satisfies
$$ \frac {|\mathcal S|}{v!} \le N \le U \frac {|\mathcal S|}{v!} $$
where $U$ is the maximum number of automorphisms of an STS from $ \mathcal S$.
\end{proposition}
So, the formulas given in Theorems~\ref{th:N3} and~\ref{th:N2} for the number 
of STS of a prescribed $3$- or $2$-rank, respectively, immediately
give the lower bounds on the number of isomorphism classes of such systems.
(Note that these lower bounds are asymptotically the same as the bounds in
\cite[Theorems~4.4,\,4.12]{JunTon:18+}; 
the difference is that using Theorems~\ref{th:N3} and~\ref{th:N2},
we more accurately exclude from the counting the systems of smaller rank.)
To obtain an upper bound,
one can substitute the upper bounds on $U$ 
proposed in \cite[Lemma~4.1]{JunTon:18+}
and \cite[Lemma~4.9]{JunTon:18+} for the cases
of $3$-rank and $2$-rank, respectively.
Clearly, the upper bound in Proposition~\ref{p:iso}
is far from the real value, as the most of systems have less than $U$ automorphisms.
Obtaining asymptotically tight formulas for the number of 
STS$(v)$ of rank $i$, for different behavior of $i=i(v)$, 
remains an important problem in this topic.

\section{Acknowledgements}

The authors thank Vladimir Tonchev for the helpful discussion and 
the anonymous reviewers for the helpful comments. 
This research is supported by
the National Natural Science Foundation of China (61672036),
the Excellent Youth Foundation of Natural Science Foundation of Anhui Province (No.1808085J20),
the Academic Fund for Outstanding Talents in Universities (gxbjZD03),
and the 
Program of Fundamental Scientific Research of the SB RAS No I.5.1., project No 0314-2019-0016.


\providecommand\href[2]{#2} \providecommand\url[1]{\href{#1}{#1}}
  \def\DOI#1{{\small {DOI}:
  \href{http://dx.doi.org/#1}{#1}}}\def\DOIURL#1#2{{\small{DOI}:
  \href{http://dx.doi.org/#2}{#1}}}

\end{document}

The following is a PYTHON program with functions 
for calculating the number of STS in accordance with the general formulas 
in (Theorems), special formulas (Corollaries) 
and the results known before (Tonchev, Zinoviev--Zinoviev, and Zinoviev's formulas). 
\verb|N_STS_2rank(V,V-R)| returns the number of STS(V) of 2-rank R;
\verb|N_STS_3rank(V,V-R-1)| returns the number of STS(V) of 3-rank R.
\begin{verbatim}
from math import factorial
def prod(I): return I[-1]*prod(I[:-1]) if I else 1

# http://oeis.org/A002860 Number of Latin squares
N_LS=[1, 1, 2, 12, 576, 161280, 812851200, 61479419904000, 108776032459082956800, 2**35*3**8*5**2*7**2*5231*3824477, 2**43*3**10*5**4*7**2*31*37*547135293937, 2**51*3**12*5**5*7**2*11*2801*2206499*62368028479]
# http://oeis.org/A036981
N_SLS=[1,1,0,6,0,720,0,31449600,0,444733651353600,0,10070314878246926155776000,0,614972203951464612786852376432607232000]
# http://oeis.org/A030128 Number of Steiner triple systems
N_STS=[1, 1, 0, 1, 0, 0, 0, 30, 0, 840, 0, 0, 0, 1197504000, 0, 60281712691200, 0, 0, 0, 1348410350618155344199680000]

def mu(x,q=2): return (-1)**x * q**(x*(x-1)/2)


def q_factorial(n,q):
    return prod([sum([q**i for i in range(s+1)]) for s in range(n)])
    # return prod([q**(s+1)-1 for s in range(n)])/(q-1)**n # for q>1 only

def Gamma(w,i,j,q=2):
    return factorial(w/q**i)**(q**i) / ( factorial(w/q**j)**(q**j) * q**((j-i)*(j+i+1)/2) * (q-1)**(j-i) * q_factorial(j-i,q) )

def Phi(v,j):
    w=v+1
    return N_STS[w/2**j-1] \
          * N_SLS[w/2**j-1]**(2**j-1) \
           * N_LS[w/2**j]**((2**j-1)*(2**j-2)/6)

def Phi3(v,j):
    return N_STS[v/3**j]**(3**j) \
           * N_LS[v/3**j]**((3**j)*(3**j-1)/6)

def N_STS_2rank(v,i): return sum(N_STS_2rank_(v,i))

def N_STS_2rank_(v,i): # v-i - rank
    N=[]; j=i
    while (v+1) % 2**j == 0:
        N+=[Gamma(v+1,0,i) * Gamma(v+1,i,j)*mu(j-i)*Phi(v,j)]; j+=1
    return N

def N_STS_3rank(v,i): return sum(N_STS_3rank_(v,i))

def N_STS_3rank_(v,i): # v-i-1 -- rank
    N=[]; j=i
    while v % 3**j == 0:
        N+=[Gamma(v,0,i,3)*Gamma(v,i,j,3)*mu(j-i,3)*Phi3(v,j)]; j+=1
    return N

#==============================================================================
def Tonchev1(k): # [19]
    return factorial(2**k-1)*(2**((2**(k-1)-1)*(2**(k-2)-1)/3)-2**(2**(k-1)-k)) / 2**(2**(k-1)-1)/prod([2**(k-1)-2**i for i in range(k-1)])

def Zinoviev2(k): # [28]
    v=2**k-1; u=2**(k-2)-1
    Mv2=6**u * ( (2**6*3**2)**(u*(u-1)/6) - 2**(u-k+2)*3*(16**(u*(u-1)/6)-2**(u-k+2)) -2**(2*u-2*k+4) )
    return Mv2 * factorial(v) / 24**u / 6 / prod([u+1-2**i for i in range(k-2)])

def Zinoviev3_(m): # wrong formula in [26]
    v=2**m-1; u=2**(m-3)-1; k=u*(u-1)/6
    Mv3=30*(factorial(7)*6240)**u*108776032459082956800**k \
        - 210 * 5040**u * (24 * 6)**u * 576**(4*k) * 2**(u-m+3) \
        +  30 * 5040**u * (( 14 * 8**u * 16**(4*k) - 8 * 2**(u-m+3) )) * 2**(2*(u-m+3))
    return Mv3 * factorial(v) / ( factorial(8)**u * 168 * prod([u+1-2**i for i in range(m-3)]) )

def Zinoviev3(m): # corrected formula in https://arxiv.org/abs/1512.00187v2
    v=2**m-1; u=2**(m-3)-1; k=u*(u-1)/6
    Mv3=30*(factorial(7)*6240)**u*108776032459082956800**k \
        - 210 * 5040**u * (24 * 6)**u * 576**(4*k) * 2**(u-m+3) \
        +  30 * 5040**u * (( 14 * 8**u * 16**(4*k) - 8 * 2**(u-m+3) )) * 2**(2*(u-m+3))
    return Mv3 * factorial(v) / ( factorial(8)**u * factorial(7) * prod([u+1-2**i for i in range(m-3)]) )

def SXK3_0(k):
    return factorial(3**k)/2**k/3**(k*(k+1)/2)/q_factorial(k,3)

def SXK3_1(k):
    v=3**k
    return factorial(v)*(2**(v**2/27-4*v/9+1)*3**(v**2/54-7*v/18+k)-1) / \
          (( 2**k * 3**(k*(k+1)/2) * q_factorial(k-1,3) ))
#    return factorial(3**k)*(2**(1+3**(k-2)*(3**(k-1)-4))*3**(k+(3**(k-2)*(3**(k-1)-7))/2)-1)/2**k/3**(k*(k+1)/2)/q_factorial(k-1,3)

def SXK3_2(k):
    v=3**k
    O=factorial(v)/q_factorial(k-2,3)/2**(k+2)/3**(k*(k+1)/2-1)
    A=5524751496156892842531225600**(3**(k-2)*(3**(k-2)-1)/6) / 2**(4*3**(k-2)-4) / 3**(3**(k-1)-2*k+2)
    A=(2**35*3**8*5**2*7**2*5231*3824477)**(v*(v-9)/486) / (( 2**(4*v/9-4) * 3**(v/3-2*k+2) ))
    B=2**(3**(k-2)*(3**(k-1)-4)+3) * 3**(3**(k-2)*(3**(k-1)-7)/2+k-1)
    B=2**(v**2/27-4*v/9+3) * 3**(v**2/54-7*v/18+k-1)
    C=1
    return sum([O*A,-O*B,O*C])

def SXK3_7k(k):
    return factorial(7*3**k)*61479419904000**(3**k*(3**k-1)/6) / (( 2**k * 3**(k*(k+1)/2) * q_factorial(k,3) * 168**(3**k) ))

def SXK1(k):
    w=2**k
    return factorial(w) * (2**((w**2-18*w+8)/24+k)-1) / 2**(k*(k+1)/2) / q_factorial(k-1,2)

def SXK2(k):
    w=2**k
    O = factorial(w) / ( 3 * 2**((k-1)*(k+2)/2) * q_factorial(k-2,2) )
    return sum([ O* 3**((w**2-12*w+32)/48) * 2**(w**2/16-5*w/4+2*k-1),  -O* 3*2**( (w**2-18*w-16)/24+k) ,O ])

def SXK3(k):
    w=2**k
    O=factorial(2**k) / ( 21 * 2**(k*(k+1)/2-3) * q_factorial(k-3,2) )
    O=factorial(w) / ( 21 * 2**(k*(k+1)/2-3) * q_factorial(k-3,2) )
    D=780**(2**(k-3)-1) * 108776032459082956800**( (2**(k-3)-1)*(2**(k-3)-2)/6 ) * 2**(3*k-12)
    D=780**(w/8-1) * (2**28*3**5*5**2*7**2*1361291)**( (w**2+128)/384-w/16 ) * 2**(3*k-12)
    E=7*2**(w**2/16-5*w/4+2*k-3) \
       * 3**((w**2-12*w+32)/48)
    F=7*2**((w**2-18*w-40)/24+k)
    G= 1
    return sum([O*D,-O*E,O*F,-O*G])

def SXK10k(k):
    w=2**k
    return factorial(10*w)*122556672**(w-1)*(2**43*3**10*5**4*7**2*31*37*547135293937)**((w-1)*(w-2)/6) \
         / 2**(k*(k+1)/2+5) / 135 / q_factorial(k,2)



for k in [3,4,5,6,7,8]: 
  print ("2-rank=v-k+3; k="+str(k)+": "+str(N_STS_2rank(2**k-1,k-3)==SXK3(k))+str(SXK3(k)==Zinoviev3_(k))+str(SXK3(k)==Zinoviev3(k)))
  if k<7: 
    print (N_STS_2rank(2**k-1,k-3))
    print (SXK3(k))
    print (Zinoviev3_(k))
    print (Zinoviev3(k))

for k in [3,4,5,6,7,8]: 
  print ("2-rank=v-k+2; k="+str(k)+": "+str(N_STS_2rank(2**k-1,k-2)==SXK2(k))+str(SXK2(k)==Zinoviev2(k)))
  if k<7: print (N_STS_2rank(2**k-1,k-2))
  if k<7: print (SXK2(k))
  if k<7: print (Zinoviev2(k))

for k in [3,4,5,6,7,8]: 
  print ("2-rank=v-k+1; k="+str(k)+": "+str(N_STS_2rank(2**k-1,k-1)==SXK1(k))+str(SXK1(k)==Tonchev1(k)))
  if k<7: print (N_STS_2rank(2**k-1,k-1))
  if k<7: print (SXK1(k))
  if k<7: print (Tonchev1(k))

for k in [3,4,5,6,7,8]: 
  print ("3-rank=v-1-k+1; k="+str(k)+": "+str(N_STS_3rank(3**k,k-1)==SXK3_1(k)))
  if k<5: print (N_STS_3rank(3**k,k-1))
  
for k in [3,4,5,6,7]: 
  print ("3-rank=v-1-k+2; k="+str(k)+": "+str(N_STS_3rank(3**k,k-2)==SXK3_2(k)))
  if k<5: print (N_STS_3rank(3**k,k-1))

for k in [2,3,4,5,6,7]:
    print ("2-rank=...; k="+str(k)+": "+str(N_STS_2rank(10*2**k-1,k)==SXK10k(k)))
    if k<4: print (N_STS_2rank(10*2**k-1,k))
\end{verbatim}

\end{document}